\documentclass[11pt]{amsart}
\usepackage{amssymb}
\usepackage{amsthm,amsmath}

\title{Axiomatizations of signed discrete Choquet integrals}

\author{Marta Cardin}
\address{Department of Applied Mathematics, University Ca' Foscari of Venice \\
Dorsoduro 3825/E--30123, Venice, Italy} \email{mcardin[at]unive.it }

\author{Miguel Couceiro}
\address{Mathematics Research Unit, FSTC, University of Luxembourg \\
6, rue Coudenhove-Kalergi, L-1359 Luxembourg, Luxembourg}%
\email{miguel.couceiro[at]uni.lu }

\author{Silvio Giove}
\address{Department of Applied Mathematics, University Ca' Foscari of Venice \\
Dorsoduro 3825/E--30123, Venice, Italy} \email{sgiove[at]unive.it }

\author{Jean-Luc Marichal}
\address{Mathematics Research Unit, FSTC, University of Luxembourg \\
6, rue Coudenhove-Kalergi, L-1359 Luxembourg, Luxembourg}%
\email{jean-luc.marichal[at]uni.lu }

\date{October 21, 2010}

\begin{document}

\theoremstyle{plain}
\newtheorem{theorem}{Theorem}[section]
\newtheorem{lemma}[theorem]{Lemma}
\newtheorem{proposition}[theorem]{Proposition}
\newtheorem{corollary}[theorem]{Corollary}
\newtheorem{fact}[theorem]{Fact}
\newtheorem*{main}{Main Theorem}

\theoremstyle{definition}
\newtheorem{definition}[theorem]{Definition}
\newtheorem{example}[theorem]{Example}

\theoremstyle{remark}
\newtheorem*{conjecture}{\indent Conjecture}
\newtheorem{remark}{\indent Remark}
\newtheorem*{claim}{Claim}

\newcommand{\N}{\mathbb{N}}                     
\newcommand{\R}{\mathbb{R}}                     
\newcommand{\Vspace}{\vspace{2ex}}                  
\newcommand{\bfx}{\mathbf{x}}

\begin{abstract}
We study the so-called signed discrete Choquet integral (also called non-monotonic discrete Choquet integral) regarded as the Lov\'asz extension
of a pseudo-Boolean function which vanishes at the origin. We present axiomatizations of this generalized Choquet integral, given in terms of
certain functional equations, as well as by necessary and sufficient conditions which reveal desirable properties in aggregation theory.
 \end{abstract}



\maketitle

\section{Introduction}

This paper deals with the so-called ``signed (discrete) Choquet integral" (also called non-monotonic Choquet integral) which naturally
generalizes the Choquet integral \cite{Cho53}. Traditionally, the  Choquet integral is defined in terms of a  capacity (also called fuzzy
measure \cite{Sug74,Sug77}), i.e., a set function $\mu\colon 2^{[n]}\to\R$ such that $\mu(\varnothing)=0$ and $\mu(S)\leqslant \mu(T)$ whenever
$S\subseteq T$. Dropping the monotonicity requirement in the definition of $\mu$, we obtain what is referred to as a signed capacity (also
called non-monotonic fuzzy measure). The signed Choquet integral is then defined exactly the same way but replacing the underlying capacity by a
signed capacity. This extension has been considered by several authors, e.g., \cite{DeWWak01,MurSugMac94,Sch86}.

A convenient way to introduce the signed Choquet integral is via the notion of Lov\'asz extension. Indeed, the signed Choquet integral can be
thought of as the Lov\'asz extension of a pseudo-Boolean function $f\colon\{0,1\}^n\to\R$ which vanishes at the origin. Moreover, we retrieve
the classical Choquet integral by further assuming that $f\colon\{0,1\}^n\to\R$  is nondecreasing.

In this paper we consider the latter approach to the signed Choquet integral. In Section 2 we recall the basic notions and terminology
concerning Choquet integrals and Lov\'asz esxtensions needed throughout the paper. In Section 3 we present various characterizations of the
signed Choquet integral. First, we recall the piecewise linear nature of Lov\'asz extensions which particularizes to the signed Choquet integral
(Theorem~\ref{Thm:3.1}). Then we generalize Schmeidler's axiomatization of the signed discrete Choquet integral given in terms of continuity and
comonotonic additivity, showing that positive homogeneity can be replaced for continuity (Theorem~\ref{Thm:3.2}). The main result of this paper,
Theorem~\ref{thm:sadf86}, presents a characterization of families of signed Choquet integrals in terms of necessary and sufficient conditions
which:
\begin{enumerate}
\item reveal the linear nature of these generalized Choquet integrals  with respect to the underlying signed capacities, \item express
properties of the family members defined on the standard basis of signed capacites, and \item make apparent the meaningfulness with respect to
interval scales of signed Choquet integrals.
\end{enumerate}
We also discuss the independence of axioms given in Theorem~\ref{thm:sadf86}.

Throughout this paper, the symbols $\wedge$ and $\vee$ denote the minimum and maximum functions, respectively.

\section{Choquet integrals and Lov\'asz extensions}

A \emph{capacity on $[n]$} is a set function $\mu\colon 2^{[n]}\to\R$ such that $\mu(\varnothing)=0$ and $\mu(S)\leqslant \mu(T)$ whenever
$S\subseteq T$. A capacity $\mu$ on $[n]$ is said to be \emph{normalized} if $\mu([n])=1$.

\begin{definition}
Let $\mu$ be a capacity on $[n]$ and let $\bfx\in [0,\infty[^n$. The \emph{Choquet integral} of $\bfx$ with respect to $\mu$ is defined by
$$
C_{\mu}(\bfx)=\sum_{i=1}^n(\mu_i^{\pi}-\mu_{i+1}^{\pi})\, x_{\pi(i)},
$$
where $\pi$ is a permutation on $[n]$ such that $x_{\pi(1)}\leqslant\cdots\leqslant x_{\pi(n)}$ and $\mu_i^{\pi}=\mu(\{\pi(i),\ldots,\pi(n)\})$
for $i\in [n+1]$, with the convention that $\mu_{n+1}^{\pi}=\mu(\varnothing)$.
\end{definition}

The concept of Choquet integral can be formally extended to more general set functions and $n$-tuples of $\R^n$ as follows. A \emph{signed
capacity} (or \emph{game}) on $[n]$ is a set function $v\colon 2^{[n]}\to\R$ such that $v(\varnothing)=0$.

\begin{definition}
Let $v$ be a signed capacity on $[n]$ and let $\bfx\in\R^n$. The \emph{signed Choquet integral} of $\bfx$ with respect to $v$ is defined by
$$
C_{v}(\bfx)=\sum_{i=1}^n(v_i^{\pi}-v_{i+1}^{\pi})\, x_{\pi(i)},
$$
where $\pi$ is a permutation on $[n]$ such that $x_{\pi(1)}\leqslant\cdots\leqslant x_{\pi(n)}$ and $v_i^{\pi}=v(\{\pi(i),\ldots,\pi(n)\})$ for
$i\in [n+1]$, with the convention that $v_{n+1}^{\pi}=v(\varnothing)$.
\end{definition}

The more general concept of a set function $v\colon 2^{[n]}\to\R$ (without any constraint) leads to the notion of the Lov\'asz extension of a
pseudo-Boolean function, which we now briefly describe. For general background, see \cite{Lov83,Sin84}.

Let $S_n$ denote the symmetric group on $[n]$ and, for each $\pi\in S_n$, define
$$
P_{\pi}=\{\bfx\in \R^n:x_{\pi(1)}\leqslant\cdots\leqslant x_{\pi(n)}\}.
$$
Let $v\colon 2^{[n]}\to\R$ be a set function and let $f\colon\{0,1\}^n\to\R$ be the corresponding pseudo-Boolean function, that is, such that
$f(\mathbf{1}_S)=v(S)$. The \emph{Lov\'asz extension} of $f$ is the continuous function $\hat f\colon \R^n\to\R$ which is defined on each
$P_{\pi}$ as the unique affine function that coincides with $f$ at the $n+1$ vertices of the standard simplex $[0,1]^n\cap P_{\pi}$ of
$[0,1]^n$. In fact, $\hat f$ can be expressed as
\begin{equation}\label{eq:1}
\hat f(\bfx)=f(\mathbf{0})+\sum_{i=1}^n(f_i^{\pi}-f_{i+1}^{\pi})\, x_{\pi(i)}\qquad (\bfx\in P_{\pi}).
\end{equation}
where $f_i^{\pi}=f(\mathbf{1}_{\{\pi(i),\ldots,\pi(n)\}})=v(\{\pi(i),\ldots,\pi(n)\})$ for $i\in [n]$ and $f_{n+1}^{\pi}=f(\mathbf{0})$. Thus
$\hat f$ is a continuous function whose restriction to each $P_{\pi}$ is an affine function.

It follows from (\ref{eq:1}) that the Lov\'asz extension of a pseudo-Boolean function $f\colon\{0,1\}^n\to\R$ is a signed Choquet integral if
and only if $f(\mathbf{0})=0$. Its restriction to $[0,\infty[^n$ is a Choquet integral if, in addition, $f$ is nondecreasing.

It was also shown \cite{Mar02b} that the Lov\'asz extension $\hat f$ can also be written as
\begin{equation}\label{eq:LovExtMob}
\hat f(\bfx)=\sum_{S\subseteq [n]}m(S)\bigwedge_{i\in S}x_i\qquad (\bfx\in \R^n),
\end{equation}
where the set function $m\colon 2^{[n]}\to\R$ is the M\"obius transform of $v$, given by $m(S)=\sum_{T\subseteq S}(-1)^{|S|-|T|}\, v(T)$. Thus,
a signed Choquet integral has the form (\ref{eq:LovExtMob}) with $m(\varnothing)=0$.

\section{Axiomatizations of Lov\'asz extensions}

We have a first characterization that immediately follows from the definition of Lov\'asz extensions.

\begin{theorem}\label{Thm:3.1}
A function $g\colon\R^n\to\R$ is a Lov\'asz extension if and only if
\begin{equation}\label{eq:ConvComb}
g(\lambda\bfx+(1-\lambda)\bfx')=\lambda\, g(\bfx)+(1-\lambda)\, g(\bfx')\qquad (0\leqslant\lambda\leqslant 1)
\end{equation}
for all comonotonic vectors $\bfx,\bfx'\in\R^n$. The function $g$ is a signed Choquet integral if additionally $g(\mathbf{0})=0$.
\end{theorem}

\begin{proof}
The condition stated in the theorem means that $g$ is affine (since it is both convex and concave) on each $P_{\pi}$. Hence, it is continuous on
$\R^n$ and thus it is a Lov\'asz extension.
\end{proof}

The following theorem is inspired from a characterization of the Choquet integral by de Campos and Bola\~{n}os \cite{deCBol92}.

\begin{theorem}\label{Thm:3.2}
A function $g\colon\R^n\to\R$ is a Lov\'asz extension if and only if the function $h\colon\R^n\to\R$, defined by $h=g-g(\mathbf{0})$,
\begin{enumerate}
\item[(i)] is comonotonic additive.

\item[(ii)] is continuous or satisfies $h(r\bfx)=r h(\bfx)$ for all $r>0$.
\end{enumerate}
The function $g$ is a signed Choquet integral if additionally $g(\mathbf{0})=0$.
\end{theorem}

\begin{proof}
It is not difficult to see that the conditions are necessary. So let us prove the sufficiency. Fix $\pi\in S_n$ and $\bfx\in P_{\pi}$. Then we
have
$$
\bfx=x_{\pi(1)}\mathbf{1}_{[n]}+\sum_{i=2}^n (x_{\pi(i)}-x_{\pi(i-1)})\mathbf{1}_{\{\pi(i),\ldots,\pi(n)\}}.
$$
By comonotonic additivity, we get
$$
h(\bfx)=h\big(x_{\pi(1)}\mathbf{1}_{[n]}\big)+\sum_{i=2}^n h\big((x_{\pi(i)}-x_{\pi(i-1)})\mathbf{1}_{\{\pi(i),\ldots,\pi(n)\}}\big).
$$
Also by comonotonic additivity, we have
$$
0=h(\mathbf{0})=h\big(\mathbf{1}_{[n]}-\mathbf{1}_{[n]}\big)=h\big(\mathbf{1}_{[n]}\big)+h\left(-\mathbf{1}_{[n]}\right)
$$
and hence $h\left(-\mathbf{1}_{[n]}\right)=-h\big(\mathbf{1}_{[n]}\big)$. Moreover, if $h(r\bfx)=r h(\bfx)$ for all $r>0$ (and even for $r=0$
since $h(\mathbf{0})=0$), then $h\left(r\mathbf{1}_{[n]}\right)=r h\left(\mathbf{1}_{[n]}\right)$ for all $r\in\R$ and hence
\begin{eqnarray*}
h(\bfx)&=& x_{\pi(1)}\, h\big(\mathbf{1}_{[n]}\big)+\sum_{i=2}^n (x_{\pi(i)}-x_{\pi(i-1)})h\big(\mathbf{1}_{\{\pi(i),\ldots,\pi(n)\}}\big)\\
&=& \sum_{i=1}^n(h_i^{\pi}-h_{i+1}^{\pi})\, x_{\pi(i)}
\end{eqnarray*}
where $h_i^{\pi}=h(\mathbf{1}_{\{\pi(i),\ldots,\pi(n)\}})$ for $i\in [n]$ and $h_{n+1}^{\pi}=h(\mathbf{1}_{\varnothing}).$

Let us now show that $h$ satisfies the positive homogeneity property as soon as it is continuous. Comonotonic additivity implies that
$g(n\bfx)=ng(\bfx)$ for every $\bfx\in\R^n$ and every positive integer $n$. For any positive integers $n,m$, we then have
$$
\frac mn\, h(\bfx)=\frac mn\, h\left(n\,\frac{\bfx}n\right)=m\, h\left(\frac{\bfx}n\right)=h\left(\frac mn\,\bfx\right)
$$
which means that $h(r\bfx)=r h(\bfx)$ for every positive rational $r$ and even for every positive real $r$ by continuity.
\end{proof}

In the following characterization of the signed Choquet integral, we will assume that the function to axiomatize is constructed from a signed
capacity. More precisely, denoting the set of signed capacities on $[n]$ by $\Sigma_n$, we now regard our function as a map
$f\colon\R^n\times\Sigma_n\to\R$, or equivalently, as the class $\{f_v\colon\R^n\to\R : v\in\Sigma_n\}$. We will adopt the latter terminology to
state our result, which is inspired from a characterization given in \cite{Mar00g}.

For every $T\subseteq [n]$, let $v_T\in\Sigma_n$ be the \emph{unanimity game} defined by $v_T(S)=1$, if $S\supseteq T$, and $0$, otherwise. Note
that the $v_T$ $(T\subseteq [n])$ form a basis (actually, the standard basis) for $\Sigma_n$. Indeed, for every $v\in\Sigma_n$, we have
$$
v=\sum_{T\subseteq [n]}m_v(T)\, v_T,
$$
where $m_v$ is the M\"obius transform of $v$.

\begin{theorem}\label{thm:sadf86}
If the class $\{f_v\colon\R^n\to\R : v\in\Sigma_n\}$ satisfies the following properties
\begin{enumerate}
\item[$(i)$] There exist $2^n$ functions $g_T\colon\R^n\to\R$ $(T\subseteq [n])$ such that
$$
f_v=\sum_{T\subseteq [n]}v(T)\, g_T\, ;
$$

\item[$(ii)$] For every $S\subseteq [n]$, we have $f_{v_S}(\bfx)=0$ whenever $x_i=0$ for some $i\in S$;

\item[$(iii)$] For every $S\subseteq [n]$, $r>0$, $s\in\R$, and $\bfx\in\R^n$, we have
$$
f_{v_S}(r\bfx+s\mathbf{1}_{[n]})=rf_{v_S}(\bfx)+s\, ;
$$
\end{enumerate}
then and only then $f_v=C_v$ for all $v\in\Sigma_n$.
\end{theorem}

\begin{proof}
The sufficiency is straightforward, so let us prove the necessity. Given the relation between $v$ and $m_v$, condition $(i)$ is equivalent to
assuming the existence of $2^n$ functions $h_T\colon\R^n\to\R$ $(T\subseteq [n])$ such that
$$
f_v=\sum_{T\subseteq [n]}m_v(T)\, h_T.
$$
Thus $f_{v_T}=h_T$. Therefore, it suffices to prove the following claim.

\begin{claim}
For any fixed $T\subseteq [n]$, if the function $f_{v_T}\colon\R^n\to\R$ satisfies conditions $(ii)$ and $(iii)$, then
$f_{v_T}(\bfx)=\wedge_{i\in T}x_i$ for all $\bfx\in\R^n$.
\end{claim}
Let $\bfx\in\R^n$. If $x_1=\cdots =x_n$, then
$$
f_{v_T}(\bfx)=f_{v_T}\Big(\big(\bigwedge_{i\in [n]}x_i\big)\mathbf{1}_{[n]}\Big)=\bigwedge_{i\in [n]}x_i\, ,
$$
since $f_{v_T}(\mathbf{0})=0$ by $(iii)$.

Otherwise, if $\bigvee_{i\in [n]}x_i-\bigwedge_{i\in [n]}x_i \neq 0$, then by $(iii)$ we have
\begin{equation}\label{eq:3.3}
f_{v_T}(\bfx)=\textstyle{\Big(\bigvee_{i\in [n]}x_i-\bigwedge_{i\in [n]}x_i\Big)\, f_{v_T}(\bfx')+\bigwedge_{i\in [n]}x_i},
\end{equation}
where
$$
\bfx'=\frac{\bfx-\big(\bigwedge_{i\in [n]}x_i\big)\mathbf{1}_{[n]}}{\bigvee_{i\in [n]}x_i-\bigwedge_{i\in [n]}x_i}\in [0,1]^n.
$$
 By $(iii)$ and $(ii)$,
$$
f_{v_T}(\bfx')=\textstyle{f_{v_T}\Big(\bfx'-\big(\bigwedge_{i\in T}x'_i\big)\mathbf{1}_{[n]}\Big)+\bigwedge_{i\in T}x'_i}=\bigwedge_{i\in
T}x'_i.
$$
By (\ref{eq:3.3}), $f_{v_T}(\bfx)=\bigwedge_{i\in T}x_i$.
\end{proof}

Note that the conditions of Theorem~\ref{thm:sadf86} are independent. Indeed,
\begin{description}
\item[$(i),(iii)\not\Rightarrow (ii)$] Consider the class $\{f_v\colon\R^n\to\R : v\in\Sigma_n\}$ given by the weighted arithmetic mean
functions
$$
f_v(\bfx)=\sum_{T\subseteq [n]} m_v(T)\,\Big(\frac{1}{|T|}\sum_{i\in T} x_i\Big),
$$
where $m_v$ is the M\"obius transform of $v$.

\item[$(i),(ii)\not\Rightarrow (iii)$] Consider the class $\{f_v\colon\R^n\to\R : v\in\Sigma_n\}$ given by the multilinear polynomial functions
$$
f_v(\bfx)=\sum_{T\subseteq [n]} m_v(T)\,\prod_{i\in T} x_i\, ,
$$
where $m_v$ is the M\"obius transform of $v$.

\item[$(ii),(iii)\not\Rightarrow (i)$] Define the normalized capacity $v^*\in\Sigma_3$ by $v^*(\{1,2\})=v^*(\{3\})=0$ and
$v^*(\{1,3\})=v^*(\{2,3\})=1/2$ and consider the class $\{f_v\colon\R^3\to\R : v\in\Sigma_3\}$ given by $f_v=C_v$ for every
$v\in\Sigma_3\setminus\{v^*\}$, and
$$
f_{v^*}(x_1,x_2,x_3)=\Big(\frac{x_1+x_2}2\Big)\wedge x_3.
$$
\end{description}

\begin{remark}
\begin{enumerate}
\item[(a)] The conditions in Theorem~\ref{thm:sadf86} can be justified as follows. Condition $(i)$ expresses the fact that the aggregation model
is linear with respect to the underlying signed capacities. Condition $(ii)$ expresses minimal requirements on the functions defined on the
standard basis $\{v_S:S\subseteq [n]\}$ of $\Sigma_n$. Condition $(iii)$ expresses the fact that $f_{v_S}$ is meaningful with respect to
interval scales.

\item[(b)] The characterization given in Theorem~\ref{thm:sadf86} does not use the fact that $v(\varnothing)=0$. Therefore they can be
immediately adapted to Lov\'asz extensions by redefining $\Sigma_n$ as the set of set functions on $[n]$.
\end{enumerate}
\end{remark}





\end{document}